\documentclass[12pt]{amsart}
\usepackage{latexsym}
\usepackage{amsmath,amssymb,amsfonts}
\usepackage{hyperref}
\usepackage{mathrsfs}

\usepackage{delarray}

\numberwithin{equation}{section}
\theoremstyle{plain}
\newtheorem{thm}{Theorem}[section]
\newtheorem{lem}[thm]{Lemma}
\newtheorem{prop}[thm]{Proposition}

\theoremstyle{definition}
\newtheorem{defn}[thm]{Definition}
\newtheorem{remark}[thm]{Remark}

\newcommand\R{{\mathbb R}}

\newcommand\N{{\mathbb N}}

\newcommand\Rn{{{\mathbb R}^n}}

\newcommand\va{\varphi}
\newcommand\pa{\partial}

\DeclareMathOperator\loc{loc}
\DeclareMathOperator\Lip{Lip}
\DeclareMathOperator\real{Re}

\title[Lower bound for the life span of the Kirchhoff equation]{A lower bound for the life span of solutions to the Kirchhoff equation with Gevrey data}

\author[Tokio Matsuyama]{Tokio Matsuyama}
\address{
Tokio Matsuyama:
 \endgraf
Department of Mathematics
\endgraf
Chuo University
\endgraf
1-13-27, Kasuga, Bunkyo-ku
\endgraf
Tokyo 112-8551
\endgraf
Japan
\endgraf
{\it E-mail address}: {\rm tokio@math.chuo-u.ac.jp}
}
\author[Lenny Neyt]{Lenny Neyt}
\address{
  Lenny Neyt:
  \endgraf
Department of Mathematics: Analysis, Logic and Discrete Mathematics
\endgraf
Ghent University
\endgraf
 Belgium  
\endgraf
{\it E-mail address}: {\rm Lenny.Neyt@UGent.be}
  }

\thanks{
T. Matsuyama was supported by Grant-in-Aid for Scientific Research (C) (No. 18K03377), Japan Society for the Promotion of Science.}

\thanks{
L. Neyt gratefully acknowledges support by FWO-Vlaanderen through the postdoctoral grant 12ZG921N}


\begin{document} 

\begin{abstract}
We provide a new lower bound for the life span of solutions to the Kirchhoff equation for which the initial data belongs to the Gevrey space. This lower bound strictly improves the classical one in the case when the frequency spectrum of the initial data is concentrated at the origin. 
\end{abstract}

\maketitle

\section{Introduction}
\label{sec:sec 1}
In this article, we concern ourselves with Kirchhoff-type equations of the form
\begin{equation}\label{EQ:Kirchhoff}
\left\{
\begin{aligned}
& \pa^2_t u-\va\left(\int_\Rn |\nabla u|^2\, dx \right)\Delta u=0, 
& \quad t>0, \quad x\in \Rn,\\
& u(0,x)=u_0(x), \quad \pa_t u(0,x)=u_1(x), &\quad x \in \mathbb{R}^n,
\end{aligned}\right.
\end{equation}
where we always assume that $\va(\rho)$ is a locally Lipschitz function on $[0,\infty)$ for which there exists a real $\nu_0>0$ such that 
\begin{equation}\label{EQ:lower bound}
\text{$\va(\rho)\ge \nu_0$ \quad for all $\rho\ge0$.}
\end{equation}
In 1876, Kirchhoff \cite{Kirchhoff} proposed the special case of
\[
n=1, \quad \va(\rho)=\nu_0+a\rho \quad (\nu_0,a>0),
\]
for the equation \eqref{EQ:Kirchhoff} to describe the transversal motions of the elastic string.
When looking at the general case, several authors have investigated the global existence for the Kirchoff-type equations when the initial data is real analytic.
In 1940, Bernstein \cite{Bernstein} first studied the global existence for analytic data in one space dimension. 
After him, in 1975, Pohozaev \cite{Pohozhaev} extended Bernstein's result to several space dimensions. 
Later, the global solvability in the real analytic class was studied by D'Ancona and Spagnolo \cite{Dancona-Invent} (see also \cite{Arosio}) under the additional assumption that  
\[
\text{$\va$ is continuous on $[0,\infty)$,}  \quad
\va(\rho)\ge 0 ,
\qquad \text{for all } \rho\geq 0 .
\]
Kajitani and Yamaguti \cite{Kajitani-Pisa} obtained the same result under a more general principal term. 

It is of course natural to ask whether the Cauchy problem \eqref{EQ:Kirchhoff} admits a unique global solution with initial data in larger function spaces, such as e.g. the quasi-analytic class or Sobolev spaces. 
The global solvability for quasi-analytic data was studied by Nishihara \cite{Nishihara} and Ghisi and Gobbino \cite{Ghisi}. Manfrin \cite{Manfrin-JDE} discovered spectral gap data which assure global solvability of the Kirchhoff equation. It should be noted that the space in \cite{Ghisi,Nishihara} is included in the Gevrey spaces. 

It has been a long-standing open problem whether or not, one can prove the existence of time global solutions in the Sobolev spaces 
\[
H^\sigma(\Rn)=(1-\Delta)^{-\frac{\sigma}{2}}L^2(\Rn) , \qquad \sigma \geq 1, 
\] 
without smallness condition on the initial data. In fact, the existence of local solutions in low regular Sobolev spaces, say, $H^\sigma\times H^{\sigma-1}$, $\sigma\in [1,3/2)$, is still not known. The main idea of the proof of the global existence of high regular solutions is to obtain boundedness of the local solutions in the $H^{3/2}$-norm at the life span. On the one hand, the main difficulty lies in controlling an intensive oscillation of the coefficient $\va(\|\nabla u(t)\|^2_{L^2})$. On the other hand, when the data is very small, one can overcome such an oscillation problem to get global solutions (see \cite{MR-Liouville} and the references therein). 
However, if one does not impose extra conditions, no results have been obtained as of yet.

As an intermediate step before considering the global solvability of the Kirchhoff equation, it is interesting to look at the existence of a life span with respect to certain initial data.
In \cite{Arosio-Garavaldi} it is shown that for any nontrivial $(u_0, u_1) \in H^{\sigma}(\Rn) \times H^{\sigma}(\Rn)$, $\sigma \geq 3/2$, there exists a life span $T_m = T_m(u_0, u_1) > 0$ such that \eqref{EQ:Kirchhoff} admits a unique maximal solution $u(t, x) \in \bigcap_{j=0,1} C^{j}([0, T_m); H^{\sigma - 1}(\Rn))$. Note that $T_m = + \infty$ corresponds to \eqref{EQ:Kirchhoff} being globally solvable for the initial data $(u_0, u_1)$. Now, if we put
	\begin{gather}
		\label{eq:Lambda}
		\Lambda :=  \nu_{0}^{-1} \left( \int_{0}^{\|\nabla u_{0}\|_{L^{2}}^{2}} \varphi(\rho) d\rho +  \| \partial_{t} u_{1} \|_{L^{2}}^{2} \right)  , \\   
		\label{eq:MandL}
		M := \sup_{\rho \in [0, \Lambda]} \varphi(\rho) , \quad \text{and} \quad L := \sup_{\rho_{1}, \rho_{2} \in [0, \Lambda]} \frac{|\varphi(\rho_{2}) - \varphi(\rho_{1})|}{|\rho_{2} - \rho_{1}|} ,
	\end{gather} 
then the following classical lower bound for $T_m$ was found in \cite[Equation (2.13)]{Arosio-Garavaldi}:
	\begin{equation} \label{eq:classical lower bound}
		T_m \geq \frac{\nu_0^{3/2}}{4L \mathcal{E}_{3/2}(u; 0)} ,
	\end{equation}
where $\mathcal{E}_{3/2}(u; t)$ is the energy of order $3/2$ of the solution (see \eqref{eq:energy of order 3/2}).

In this paper, we will consider the case where the initial data is contained in the Gevrey spaces, which lie in between the real analytic class and the Sobolev spaces. For $s\geq1$, we denote by $\gamma^s_{L^2}(\Rn)$ the Roumieu-Gevrey space of order $s$ on $\Rn$, 
\[
\gamma^s_{L^2}(\Rn)=\bigcup_{\eta>0}\gamma^s_{\eta,L^2}(\Rn),
\]
endowed with its natural $(LB)$-space topology, where $f$ belongs to $\gamma^s_{\eta,L^2}(\Rn)$ if 
\[
\|f\|_{\gamma^s_{\eta,L^2}}
=\left(\int_\Rn e^{\eta|\xi|^{\frac{1}{s}}}|(\mathcal{F}f)(\xi)|^2\, d\xi\right)^{\frac12} < \infty ;
\]
here $(\mathcal{F}f)(\xi)$ stands for the Fourier transform of $f(x)$.
If $f, g \in \gamma^s_{\eta,L^2}$, we also consider the norm
\[
\|(f,g)\|_{\gamma^s_{\eta,L^2}\times \gamma^s_{\eta,L^2}}
=\sqrt{\|f\|^2_{\gamma^s_{\eta,L^2}}+\|g\|^2_{\gamma^s_{\eta,L^2}}} .
\]
Note that in the particular case $s=1$, $\gamma^1_{L^2}(\Rn)$ is exactly the real analytic class, and its global existence was proved by Bernstein \cite{Bernstein} for $n = 1$ and by Pohozaev \cite{Pohozhaev} for $n \geq 2$. For $s > 1$, the well-posedness of the Kirchhoff equation with initial data in $\gamma^{s}_{L^{2}}(\Rn)$ was first considered in \cite{MR-JAM}. Here, we will provide an explicit lower bound for $T_m$ in function of the Gevrey norm of the initial data. In fact, we have the following result.

\begin{thm} \label{thm:Lower bound Tm}
Suppose that $\va(\rho)$ is a locally Lipschitz function on $[0,\infty)$ satisfying the non-degeneracy condition \eqref{EQ:lower bound}.
Let $s>1$ and suppose $(u_0,u_1) \in \gamma_{L^2}^{s}(\Rn) \times \gamma_{L^2}^{s}(\Rn)$. If, for  $\eta > 2M\nu_{0}^{-1}$, we have $((-\Delta)^{\frac{3}{4}} u_{0}, (-\Delta)^{\frac{1}{4}} u_{1}) 
\in \gamma_{\eta, L^{2}}^{s}(\Rn) \times \gamma_{\eta, L^{2}}^{s}(\Rn)$, then, we have the lower bound
	\begin{equation}
		\label{eq:Lower bound Tm}
		T_{m} \geq \left[ \frac{\min(\nu_0, 1)}{\max(M, 1)} \frac{e^{-2 \nu_0^{-1} M}}{2sL} \frac{\nu_0 \eta - 2M}{\left\| \left( (-\Delta)^{\frac{3}{4}} u_{0}, (-\Delta)^{\frac{1}{4}} u_{1} \right) \right\|_{\gamma_{\eta, L^{2}}^{s} \times \gamma_{\eta, L^{2}}^{s}}^{2}} \right]^{\frac{s}{s + 1}} .
	\end{equation}
\end{thm}

Depending on the data, the lower bound given in \eqref{eq:Lower bound Tm} will be strictly larger than the one classically given by \eqref{eq:classical lower bound} (see Remark \ref{remark: improvement lower bound}). This seems to be especially the case when the frequency spectrum of the initial data is concentrated at the origin. Moreover, we also mention that our proof could be adapted, similarly as in \cite{MR-JAM}, to find an analogous result for the initial-boundary value problems of the Kirchhoff equation with initial data in the Gevrey class. We have organized the paper as follows: We first state some known results on local existence theorems in Section \ref{sec:local existence theorems}, after which we prove our main result in Section \ref{sec:proof of main theorem}. 

\section{Local existence theorems}
\label{sec:local existence theorems}

In the context of the Sobolev spaces, the Kirchhoff equation has a first integral.

\begin{lem} \label{lem:Energy}
Let $T>0$. Assume that, for some $\sigma \geq 3/2$, $u\in \displaystyle{\bigcap_{j=0}^1}\, C^j([0,T];H^{\sigma-j}(\Rn))$ is the solution to \eqref{EQ:Kirchhoff}.
If we define the energy
\[
\mathcal{H}(u;t):=
\|\pa_t u(t)\|^2_{L^2}+\int^{\|\nabla u(t)\|^2_{L^2}}_0 \va(\rho)\, d\rho ,
\]
then, we have 
\begin{equation}\label{eq.cons}
\mathcal{H}(u;t)=\mathcal{H}(u;0) \quad \text{for all $t\in [0,T]$}.
\end{equation}
\end{lem}

\begin{proof}
The proof is straightforward:  Multiplying \eqref{EQ:Kirchhoff} by $\pa_t u$ and integrating over $\Rn$ gives
\[
\frac{d}{dt}\mathcal{H}(u;t)=0,
\]
as desired.
\end{proof}

For $\sigma \in \R$, we denote the homogeneous counterpart of the fractional Sobolev spaces by
\[
\dot{H}^\sigma(\Rn)=(-\Delta)^{-\frac{\sigma}{2}}L^2(\Rn).
\]
We now define the energy of order $3/2$ for any $u\in {\bigcap_{j=0}^1}\, C^j([0,T]; \dot{H}^{\sigma-j}(\Rn))$ as follows:
\begin{equation} \label{eq:energy of order 3/2}
	\mathcal{E}_{3/2}(u;t) = \va \left(\|\nabla u(t)\|^2_{L^2}\right) \|u(t)\|^2_{\dot{H}^{\frac32}} + \|\pa_t u(t)\|^2_{\dot{H}^{\frac12}}.
\end{equation}

The following result was shown by Arosio and Garavaldi.

\begin{thm}[{\cite[Theorem 2]{Arosio-Garavaldi}}] \label{thm:LocalExistenceSobolev}
Suppose that $\va(\rho)$ is a locally Lipschitz function on $[0,\infty)$ satisfying the non-degeneracy condition \eqref{EQ:lower bound}. Let $\sigma\geq 3/2$. Then for any nontrivial $(u_0,u_1)\in H^\sigma(\Rn) \times H^{\sigma-1}(\Rn)$, there exists a life span $T_m=T_m(u_0,u_1)>0$ depending only on $\mathcal{H}(u;0)$ and $\mathcal{E}_{3/2}(u;0)$ such that the Cauchy problem \eqref{EQ:Kirchhoff} admits a unique maximal solution $u(t,x)$ in the class  
$$u\in C([0,T_m);H^\sigma(\Rn)) \cap C^1([0,T_m);H^{\sigma-1}(\Rn)),$$ 
and one of the following statements is true{\rm :}
\begin{enumerate}
\item[(i)] $T_m=+\infty${\rm ;}
\item[(ii)] $T_m<+\infty$ and $\displaystyle{\limsup_{t\to T_m^{-}}}\,\mathcal{E}_{3/2}(u;t)=+\infty.$ 
\end{enumerate}
\end{thm}

We remark here that the life span $T_m$ is to be understood as follows:
\[
T_m=\sup\, \left\{t: \; \text{$H^{\frac32}$-solution $u(\tau,\cdot)$ to \eqref{EQ:Kirchhoff}
with data $(u_0,u_1)$
exists for $0\leq \tau<t$}\right\}.
\]
It should be noted that, however big the regularity of the data is, $T_m$ depends only on the norm of the data in $\dot{H}^{3/2}(\Rn)\times \dot{H}^{1/2}(\Rn)$.
This means that when one would show the global existence of solutions to \eqref{EQ:Kirchhoff}, it suffices to obtain that the norm of solutions in $\dot{H}^{3/2}(\Rn)\times \dot{H}^{1/2}(\Rn)$ is bounded on $[0,T_m)$.

The local existence theorem for Gevrey spaces is now immediately obtained as a consequence of Theorem \ref{thm:LocalExistenceSobolev}, and the life span depends only on the constants $\mathcal{H}(u;0)$ and $\mathcal{E}_{3/2}(u;0)$. More precisely, we have the following:

\begin{prop} \label{prop:G-local}
Suppose that $\va(\rho)$ is a locally Lipschitz function on $[0,\infty)$ satisfying 
the non-degeneracy condition \eqref{EQ:lower bound}.
Let $s>1$ and $\eta>0$. For any nontrivial $(u_0,u_1)\in 
(-\Delta)^{-\frac34}\gamma^s_{\eta,L^2}(\Rn)\times (-\Delta)^{-\frac14}\gamma^s_{\eta,L^2}(\Rn)$,
there exists a life span $T_m=T_m(u_0,u_1)>0$ depending only on 
$\mathcal{H}(u;0)$ and $\mathcal{E}_{3/2}(u;0)$ 
such that the Cauchy problem 
\eqref{EQ:Kirchhoff} admits a unique solution $u(t,x)$ in the class 
$$u\in C\left([0,T_m);(-\Delta)^{-\frac34}\gamma^s_{\eta,L^2}(\Rn)\right)
\cap C^1\left([0,T_m);(-\Delta)^{-\frac14}\gamma^s_{\eta,L^2}(\Rn)\right),
$$
and one of the following statements is true{\rm :}
\begin{enumerate}
\item[(i)] $T_m=+\infty${\rm ;}

\item[(ii)] $T_m<+\infty$ and $\displaystyle{\limsup_{t\to T_m^{-}}}\,
\mathcal{E}_{3/2}(u;t)=+\infty$. 
\end{enumerate}
\end{prop}

\begin{proof}
We may see the initial data in $(-\Delta)^{-\frac34}\gamma^s_{\eta,L^2}(\Rn)\times (-\Delta)^{-\frac14}\gamma^s_{\eta,L^2}(\Rn)$ as elements of the phase space $\dot{H}^{3/2}(\Rn)\times \dot{H}^{1/2}(\Rn)$. 
Let $T\in(0,T_m)$ be arbitrarily fixed. By Theorem \ref{thm:LocalExistenceSobolev}, with $\sigma=3/2$, we know that the Cauchy problem \eqref{EQ:Kirchhoff} admits a unique solution $u$ such that 
\[
u\in C([0,T];\dot{H}^{\frac32}(\Rn))\cap C^1([0,T];\dot{H}^{\frac12}(\Rn)) .
\]
Put
\[
c_u(t)=\va(\|\nabla u(t)\|^2_{L^2}) \in \Lip_{\loc}([0,T]). 
\]
It follows by the theory of linear partial differential equations that the Cauchy problem 
\[
\pa^2_t v-c_u(t)\Delta v=0, \quad t>0,  \quad x\in\Rn,
\]
with initial data $(u_0, u_1)$, admits a unique solution $v(t,x)$ such that 
\[
v\in C\left([0,T];(-\Delta)^{-\frac34}\gamma^s_{\eta,L^2}(\Rn)\right)
\cap C^1\left([0,T];(-\Delta)^{-\frac14}\gamma^s_{\eta,L^2}(\Rn)\right).
\]
Then we conclude that $v=u$, i.e., the Cauchy problem \eqref{EQ:Kirchhoff} admits a unique solution $u$ such that 
\[
u\in C\left([0,T];(-\Delta)^{-\frac34}\gamma^s_{\eta,L^2}(\Rn)\right)
\cap C^1\left([0,T];(-\Delta)^{-\frac14}\gamma^s_{\eta,L^2}(\Rn)\right). 
\]
From here the result follows.
\end{proof}

We end this section with a remark on the constants in \eqref{eq:Lambda} and \eqref{eq:MandL}.

\begin{remark} \label{remark:Lambda M and L}
Depending on the initial data $(u_{0}, u_{1}) \in \gamma^{s}_{L^{2}}(\Rn) \times \gamma^{s}_{L^{2}}(\Rn)$, the domain of $\varphi$ in \eqref{EQ:Kirchhoff} is bounded. Indeed, suppose that $u(t, x)$ is the solution to \eqref{EQ:Kirchhoff} with life span $T_m = T_m(u_0, u_1) > 0$ and let $\Lambda$ be as in \eqref{eq:Lambda}. Then, in particular,
	\[ \Lambda = \nu_{0}^{-1} \mathcal{H}(u;0) . \]
Now, it follows from \eqref{EQ:lower bound} and \eqref{eq.cons} 
\[
\|\nabla u(t,\cdot)\|^2_{L^2}\leq \nu^{-1}_0\mathcal{H}(u;t)
=\nu^{-1}_0\mathcal{H}(u;0) = \Lambda
\]
for any $t\in[0,T_m)$. This implies that $[0,\Lambda]$ is the actual domain of $\va(\rho)$ in this context. Then, if $M$ and $L$ are as in \eqref{eq:MandL}, it follows that
	\[ \nu_{0} \leq \varphi(\rho) \leq M \quad \text{for all } \rho \in [0, \Lambda] ,  \]
and
	\[ |\varphi^{\prime}(\rho)| \leq L \quad \text{for almost all } \rho \in [0, \Lambda] . \]
\end{remark}

\section{The proof of Theorem \ref{thm:Lower bound Tm}}
\label{sec:proof of main theorem}

We now focus on proving Theorem \ref{thm:Lower bound Tm}. To do this, we will consider linear Cauchy problems of the form
\begin{equation} \label{EQ:Linear}
\left\{
\begin{aligned}
& \pa^2_t v-c(t) \Delta v=0, & \quad 
t\in (0,T), \quad x\in \Rn,\\
& v(0,x)=u_0(x), \quad \pa_t v(0,x)=u_1(x), 
&\quad x \in \mathbb{R}^n.
\end{aligned}\right. 
\end{equation} 
In the case where the derivative of $c$ has a pole at $T$, we find the following result when $u_0$ and $u_1$ belong to $\gamma^{s}_{L^{2}}(\Rn)$ (see also \cite{Colombini}).

\begin{prop} 
\label{prop:G-Colombini}
Let $1/(q-1)\leq s<q/(q-1)$ and $q>1$. 
Assume that $c(t)$ is a function on $[0, T]$ that belongs to $\mathrm{Lip}_{\mathrm{loc}}([0,T))$ and satisfies 
\begin{gather}
\nu_0\leq c(t) \leq M, \quad t\in [0,T], \label{EQ:G-hyp1} \\
\left|c^\prime(t)\right|\leq \frac{K}{(T-t)^q}, \quad a.e.\, t\in [0,T) , \label{EQ:G-hyp2}
\end{gather}
for some $0<\nu_0<M$ and $K>0$.
Take any $(u_0, u_1) \in (-\Delta)^{-\sigma - 1/2} \gamma^s_{\eta, L^2}(\Rn) \times (-\Delta)^{-\sigma} \gamma^s_{\eta, L^2}(\Rn)$ for some $\sigma \geq 0$ and
\begin{equation} \label{eq.eta}
\eta>\left(\frac{K}{q-1}+2M\right)\nu^{-1}_0 .
\end{equation}
Then, the Cauchy problem \eqref{EQ:Linear} with initial data $(u_0, u_1)$ admits a unique solution $v \in C^1([0,T];\gamma^s_{L^2}(\Rn))$, and
\begin{equation}\label{EQ:G-interval}
\begin{split}
& \nu_0 \|(-\Delta)^{\sigma + 1/2} v(t)\|^2_{\gamma^s_{\eta^\prime,L^2}}
+\| \pa_t (-\Delta)^{\sigma} v(t)\|^2_{\gamma^s_{\eta^\prime,L^2}}\\
& \qquad
\leq\max(M,1)e^{2\nu^{-1}_0M \max\{1,T^{1-(qs-s)}\}}
\|((-\Delta)^{\sigma + 1/2} u_0, (-\Delta)^{\sigma} u_1)\|^2_{\gamma^s_{\eta,L^2} \times \gamma^s_{\eta,L^2}}
\end{split}
\end{equation}
for $t\in [0,T]$, where 
\[
\eta^\prime
=\eta-\left(\frac{K}{q-1}+2M\right)\nu^{-1}_0>0.
\]
\end{prop}

\begin{proof}
Suppose $((-\Delta)^{\sigma + 1/2} u_{0}, (-\Delta)^{\sigma} u_1) \in \gamma^s_{\eta, L^2}(\Rn)\times \gamma^s_{\eta, L^2}(\Rn)$ for some $\sigma \geq 0$ and $\eta$ satisfying \eqref{eq.eta}.
Let $w=w(t,\xi)$ be a solution of the Cauchy problem 
\[
\left\{
\begin{aligned}
& \partial^{2}_{t} w + c(t) |\xi|^2 w = 0, & \quad t\in (0,T), \quad \xi \in \Rn ,\\
& w(0,\xi)=(\mathcal{F}u_0)(\xi), \quad \pa_t w(0,\xi)=(\mathcal{F}u_1)(\xi) , & \quad \xi \in \Rn . 
\end{aligned}
\right.
\]
We define 
\[
c_*(t,\xi)=
\begin{cases}
c(T) & \quad \text{if $T|\xi|^{\frac{1}{qs-s}}\leq 1$,}\\
c(t) & \quad \text{if $T|\xi|^{\frac{1}{qs-s}}>1$ and $0\leq t\le T-|\xi|^{-\frac{1}{qs-s}}$,}\\
c\left(T-|\xi|^{-\frac{1}{qs-s}}\right) & \quad \text{if $T|\xi|^{\frac{1}{qs-s}}>1$ and $T-|\xi|^{-\frac{1}{qs-s}}<t\leq T$,}
\end{cases}
\]
and 
\[
\alpha(t,\xi)=\nu^{-1}_0 
|c_*(t,\xi)-c(t)||\xi|
+\frac{|\partial_{t} c_*(t,\xi)|}{c_*(t,\xi)}.
\]
We adopt an energy for $w$ as 
\[
E(t,\xi)=\left[
|\partial_{t} w(t,\xi)|^2+c_*(t,\xi) |\xi|^2|w(t,\xi)|^2
\right] |\xi|^{4 \sigma} k(t,\xi),
\]
where 
\[
k(t,\xi)=\exp\left(-\int^t_0 \alpha(\tau,\xi)\, d\tau+\eta|\xi|^{\frac{1}{s}}\right) .
\]
We put
\[
\mathcal{E}(t)=\int_\Rn E(t,\xi)\, d\xi ,
\]
and note that
\begin{equation} \label{EQ:E0}
\mathcal{E}(0) \leq \max(M, 1) \|((-\Delta)^{\sigma + 1/2} u_0, (-\Delta)^{\sigma} u_1)\|^2_{\gamma^s_{\eta,L^2} \times \gamma^s_{\eta,L^2}} .
\end{equation}
We first estimate the integral of $\alpha(t,\xi)$.
When 
$$
T|\xi|^{\frac{1}{qs-s}}\leq1,
$$ 
we find by \eqref{EQ:G-hyp1},
\begin{equation}\label{EQ:IMP1}
\int^t_0\alpha(\tau,\xi)\, d\tau
\leq \int^{T}_0 \nu^{-1}_0 |c(T)-c(\tau)||\xi|\, d\tau
\leq 2\nu^{-1}_0M T|\xi| 
\leq 2\nu^{-1}_0M T^{1-(qs-s)},
\end{equation}
while if $$T|\xi|^{\frac{1}{qs-s}}>1,$$
it follows from \eqref{EQ:G-hyp1} and \eqref{EQ:G-hyp2} that 
\begin{equation}\label{EQ:IMP2}
\begin{split}
 \int^t_0\alpha(\tau,\xi)\, d\tau
\leq&\, \int^{T-|\xi|^{-\frac{1}{qs-s}}}_0 \frac{|c^\prime(\tau)|}{c(\tau)}\, d\tau
+\int^T_{T-|\xi|^{-\frac{1}{qs-s}}} \nu^{-1}_0 |c_*(\tau,\xi)-c(\tau)||\xi|\, d\tau \\
\leq& \,\int^{T-|\xi|^{-\frac{1}{qs-s}}}_0 \frac{K \nu^{-1}_0}{(T-\tau)^q}\, d\tau
+2\nu^{-1}_0M |\xi|^{1-\frac{1}{qs-s}} \\
\leq& \, \frac{K\nu^{-1}_0|\xi|^{\frac{1}{s}}}{q-1}+2\nu^{-1}_0M |\xi|^{1-\frac{1}{qs-s}}.
\end{split}
\end{equation}
Since $1-1/(qs-s)<1/s$ by our assumptions on $s$ and $q$, it follows that
\[
|\xi|^{1-\frac{1}{qs-s}} \leq (1+|\xi|)^{\frac{1}{s}} \le 1+|\xi|^{\frac{1}{s}}.
\]
Consequently, we infer from \eqref{EQ:IMP1} and \eqref{EQ:IMP2} that
\[
k(t,\xi)\geq e^{-2\nu^{-1}_0M \max\{1,T^{1-(qs-s)}\}}
e^{\left(\eta-\frac{K\nu^{-1}_0}{q-1}-2\nu^{-1}_0M \right)|\xi|^{\frac{1}{s}}},
\]
and hence,
\begin{equation}\label{EQ:G-Energy1}
\begin{split}
 \mathcal{E}(t)\geq & \,   
e^{-2\nu^{-1}_0M \max\{1,T^{1-(qs-s)}\}} \\
& \qquad \cdot \int_\Rn e^{\left(\eta-\frac{K\nu^{-1}_0}{q-1}-2\nu^{-1}_0M \right)|\xi|^{\frac{1}{s}}}
\left[\nu_0 |\xi|^2|w(t,\xi)|^2+|\partial_{t} w(t,\xi)|^2\right] |\xi|^{4 \sigma} d\xi.
\end{split}
\end{equation}
We may compute the time derivative of $E(t,\xi)$,
\begin{align*}
\partial_{t} E(t,\xi)=
&\, \left[2\mathrm{Re} (\partial_{t}^{2} w \overline{\partial_{t} w})
+\partial_{t} c_*(t,\xi)|\xi|^2|w|^2
+2c_*(t,\xi) |\xi|^2\mathrm{Re} (\partial_{t} w \overline{w})
\right] |\xi|^{4 \sigma} k(t,\xi)\\
{} & -\{c_*(t,\xi) |\xi|^2 |w|^2+|\partial_{t} w|^2\} \alpha(t,\xi) |\xi|^{4 \sigma} k(t,\xi)\\
=&\, \left[\{c_*(t,\xi)-c(t)\}|\xi|^2\mathrm{Re} (\partial_{t} w \overline{w})
+ \partial_{t}c_*(t,\xi)|\xi|^2|w|^2\right] |\xi|^{4 \sigma} k(t,\xi) \\
{} & - \alpha(t,\xi)E(t,\xi),
\end{align*}
and note that for the left part we have
\begin{multline*}
\left[\frac{|c_*(t,\xi)-c(t)||\xi|}{c_*(t,\xi)}
|\partial_{t} w| \cdot c_*(t,\xi)|\xi| |w|
+\frac{|\partial_{t} c_*(t,\xi)|}{c_*(t,\xi)} c_*(t,\xi) |\xi|^2|w|^2\right] |\xi|^{4\sigma} k(t,\xi) \\
\leq \left[\nu^{-1}_0|c_*(t,\xi)-c(t)| |\xi| 
+\frac{|\partial_{t} c_*(t,\xi)|}{c_*(t,\xi)} \right] E(t,\xi) = \alpha(t,\xi) E(t,\xi) ,
\end{multline*}
which implies that $\pa_t E(t,\xi)\leq 0$ for a.e. $t\in[0,T]$. Consequently,
\[
\mathcal{E}(t)\leq \mathcal{E}(0),
\]
so that \eqref{EQ:G-interval} follows directly from \eqref{EQ:E0} and \eqref{EQ:G-Energy1}.
\end{proof}

For the remainder of this section, we fix
	\[ (u_0,u_1) \in (-\Delta)^{-\frac{3}{4}} \gamma_{\eta, L^2}^{s}(\Rn) \times (-\Delta)^{-\frac{1}{4}} \gamma_{\eta, L^2}^{s}(\Rn) , \] 
with $\eta > 2 M \nu_0^{-1}$. Moreover, we will assume that $T_m = T_m(u_0, u_1) <+ \infty$, as otherwise Theorem \ref{thm:Lower bound Tm} is trivial. Our proof is based on a contradiction argument, that is, we will from now on suppose that \eqref{eq:Lower bound Tm} is false and from there show that the life span is then strictly larger than $T_m$. For this, we will consider the following class of functions.

\begin{defn}
Let 
	\begin{equation}
		\label{eq:K}
		K := \frac{2L \max(M, 1)}{\min(\nu_0, 1)} e^{2 \nu_0^{-1} M} T_{m}^{\frac{s + 1}{s}} \left\| \left( (-\Delta)^{\frac{3}{4}} u_{0}, (-\Delta)^{\frac{1}{4}} u_{1} \right) \right\|_{\gamma_{\eta, L^{2}}^{s} \times \gamma_{\eta, L^{2}}^{s}}^{2} . 
	\end{equation}
We define the class $\mathcal{K}$ as all those functions $c$ on $[0, T_m]$ such that $c \in \mathrm{Lip}_{\mathrm{loc}}([0,T_m))$ for which
\[
\left\{
\begin{aligned}
& \nu_0 \leq c(t)\leq M, &\quad t \in [0, T_{m}],\\
& \left| c^\prime(t) \right|\leq \frac{K}{(T_m - t)^{\frac{s + 1}{s}}}, 
&\quad \mathrm{a.e.}\, t\in [0,T_m).
\end{aligned}\right.
\] 
We endow $\mathcal{K}$ with the topology induced by the Fr\'{e}chet space $L^{\infty}_{\loc}([0, T_m))$. 
\end{defn} 

Note that if \eqref{eq:Lower bound Tm} doesn't hold, then this implies exactly that \eqref{eq.eta} holds with $K$ as in \eqref{eq:K} and $q = (s + 1) / s$. Consequently, by Proposition \ref{prop:G-Colombini}, for any $c \in \mathcal{K}$ the Cauchy problem \eqref{EQ:Linear} with initial data $(u_0, u_1)$ has a unique solution $v(t, x) \in C^{1}\left([0, T_{m}]; \gamma^{s}_{L^{2}}(\Rn)\right)$. We now consider the function
	\[ \varphi^{*}(\rho) = \begin{cases} \varphi(\rho) , & 0 \leq \rho \leq \Lambda, \\ \varphi(\Lambda) , & \rho > \Lambda . \end{cases} \]
Then $\varphi^{*} \in \mathrm{Lip}_{\mathrm{loc}}([0,\infty))$, and note that by Remark \ref{remark:Lambda M and L}, in case of the initial data $(u_0, u_1)$, we may exchange $\varphi$ with $\varphi^{*}$ in \eqref{EQ:Kirchhoff} and obtain the same solution $u(\cdot, x)$ on $[0, T_m)$. Given a $c \in \mathcal{K}$, we define the function
	\[ c_{v}(t) := \varphi^{*}\left( \int_{\Rn} |\nabla v(t, x)|^{2} dx \right) . \]
Theorem \ref{thm:Lower bound Tm} will then follow from the following two crucial results.

	\begin{lem}
		\label{l:ThetaContinuous}
		The mapping
			\begin{equation} 
				\label{eq:Theta}
				\Theta : \mathcal{K} \rightarrow \mathcal{K} : \quad c(t) \mapsto c_{v}(t) , 
			\end{equation}
		is well-defined and continuous.
	\end{lem}
	
	\begin{lem}
		\label{l:KConvexCompact}
		$\mathcal{K}$ is a convex and compact Fr\'{e}chet space.
	\end{lem}
	
Before showing these lemmas, let us first demonstrate how they entail the proof of Theorem \ref{thm:Lower bound Tm}.

\begin{proof}[Proof of Theorem {\rm \ref{thm:Lower bound Tm}}]
By Lemmas \ref{l:ThetaContinuous} and \ref{l:KConvexCompact}, it follows from the Schauder-Tychonoff theorem that the mapping $\Theta$ in \eqref{eq:Theta} has a fixed point $c_0$ in $\mathcal{K}$. Consequently, the solution $v(t, x)$ to the Cauchy problem \eqref{EQ:Linear} with $c = c_0$ and initial data $(u_0, u_1)$ is also a solution $u = u(t, x)$ to the non-linear Cauchy problem \eqref{EQ:Kirchhoff} with initial data $(u_0, u_1)$ on $[0, T_m]$. Hence $u$ exists at the endpoint $T_m$, so that $\mathcal{E}_{3/2}(u; T_m) <+ \infty$, contradicting Proposition \ref{prop:G-local}. Therefore, we may conclude that \eqref{eq:Lower bound Tm} holds.
\end{proof}

We now move on to prove the lemmas. 

\begin{proof}[Proof of Lemma {\rm \ref{l:ThetaContinuous}}]
We first show that $\Theta$ is well-defined, i.e. that for every $c \in \mathcal{K}$ also $c_{v} \in \mathcal{K}$. It is clear that $c_{v} \in \mathrm{Lip}_{\mathrm{loc}}([0,T_m))$, and by the definition of $\varphi^{*}$ we have that 
$$\nu_{0} \leq c_{v}(t) \leq M$$ for all $t \in [0, T_{m}]$. 

For the derivative, take 
	\begin{equation} 
		\label{eq:etaprime}
		\eta' = \eta - (K s + 2M) \nu^{-1}_{0} > 0 ,
	\end{equation}
then, by Proposition \ref{prop:G-Colombini}, we deduce that almost everywhere
	\begin{align*} 
		|c'_{v}(t)| 
		&= \left| (\varphi^{*})^{\prime}(\|\nabla v(t)\|_{L^{2}}^{2}) \cdot 2 \real \left( (-\Delta)^{\frac{3}{4}} v(t), \partial_{t} (-\Delta)^{\frac{1}{4}} v(t) \right)_{L^{2}} \right| \\
		&\leq 2 |\varphi^{\prime}(\|\nabla v(t)\|_{L^{2}}^{2})| \|v(t)\|_{\dot{H}^{\frac{3}{2}}} \|\partial_{t} v(t)\|_{\dot{H}^{\frac{1}{2}}} \\
		&\leq 2 L \|(-\Delta)^{\frac{3}{4}} v(t)\|_{\gamma^{s}_{\eta', L^{2}}} \|\partial_{t} (-\Delta)^{\frac{1}{4}} v(t)\|_{\gamma^{s}_{\eta', L^{2}}} \\
		&\leq \frac{2 L \max(1, M)}{\min(1, \nu_{0})} e^{2 \nu_{0}^{-1} M} \left\| \left((-\Delta)^{\frac{3}{4}} u_{0},
		 (-\Delta)^{\frac{1}{4}} u_{1} \right) \right\|^2_{\gamma^{s}_{\eta, L^{2}} \times \gamma^{s}_{\eta, L^{2}}} \\
		&= K / T_{m}^{\frac{s + 1}{s}} .
	\end{align*}
On the other hand, it trivially holds that
	\[ 1 = \frac{T_{m}^{\frac{s + 1}{s}}}{T_{m}^{\frac{s + 1}{s}}} \leq \frac{T_{m}^{\frac{s + 1}{s}}}{(T_{m} - t)^{\frac{s + 1}{s}}} . \]
Combining these two estimates together, we find that almost everywhere
	\[ |c'_{v}(t)| \leq \frac{K}{(T_{m} - t)^{\frac{s + 1}{s}}} .  \]
Consequently, $c_{v} \in \mathcal{K}$, so that $\Theta$ is well-defined.

Next, we show that $\Theta$ is continuous. To do this, let us take a sequence $(c_k(t))_{k \in \N}$ in $\mathcal{K}$ such that 
\[ c_k(t) \to c(t) \in \mathcal{K} \quad \text{in } L^{\infty}_{\loc}([0,T_m)), \qquad k \to \infty , \]
and let $v_k(t,x)$ and $v(t,x)$ be the corresponding solutions to the linear Cauchy problem \eqref{EQ:Linear} with the coefficients $c_k(t)$ and $c(t)$, respectively. 
Then it is sufficient to  prove that the images $\tilde{c}_k(t) :=\Theta(c_k(t))$ and $\tilde{c}(t):=\Theta(c(t))$ satisfy 
	\begin{equation} 
		\label{EQ:small-S-convergence}
		\tilde{c}_k(t) \to \tilde{c}(t) \quad \text{in } L^{\infty}_{\loc}([0,T_m)), \qquad k \to \infty. 
	\end{equation} 
The functions $w_k:=v_k-v$, $k=1,2,\ldots$, solve the linear Cauchy problems
	\[ 
		\begin{cases} 
			\partial^2_t w_k-c(t) \Delta w_k = \left\{ c_k(t)-c(t) \right\} \Delta v_k, & (t,x) \in (0,T_m) \times \Rn , \\
			w_k(0,x)=0, \quad \partial_t w_k(0,x)=0, & x \in \Rn.
		\end{cases} 
	\]
We define the energies
		\[ \mathcal{E}_{w_k}(t) = \Vert \pa_t w_k(t) \Vert^{2}_{L^2} + c(t) \| \nabla w_k(t) \|^{2}_{L^2} . \]
Then, for $\eta'$ as in \eqref{eq:etaprime}, differentiating gives, by Proposition \ref{prop:G-Colombini},
	\begin{align*}
		\mathcal{E}_{w_{k}}^{\prime}(t) 
		&= 2\left\{c_k(t)-c(t) \right\} \real\left(\Delta v_k(t),\pa_t w_k(t)\right)_{L^2}+ c^{\prime}(t) \left\| \nabla w_k(t) \right\|^{2}_{L^2} \\
		&\leq 2\left\vert c_k(t)-c(t) \right\vert \| v_k(t) \|_{\dot{H}^{\frac32}} \| \pa_t w_k(t) \|_{\dot{H}^{\frac12}} + \frac{|c^{\prime}(t)|}{c(t)} \mathcal{E}_{w_k}(t) \\
		&\leq 2\left\vert c_k(t)-c(t) \right\vert \| (-\Delta)^{\frac{3}{4}} v_k(t) \|_{\gamma^{s}_{\eta', L^{2}}} \cdot \\ 
		&\hspace{2cm} \left(\| \pa_t (-\Delta)^{\frac{1}{4}} v_k(t) \|_{\gamma^{s}_{\eta', L^{2}}} + \| \pa_t (-\Delta)^{\frac{1}{4}} v(t) \|_{\gamma^{s}_{\eta', L^{2}}} \right)  + \frac{|c^{\prime}(t)|}{c(t)} \mathcal{E}_{w_k}(t) \\
		&\leq \frac{4 \max(1, M)}{\min(1, \nu_{0})} e^{2 \nu_{0}^{-1} M} \left\vert c_k(t)-c(t) \right\vert \cdot \\
		&\hspace{2cm} \left\| \left((-\Delta)^{\frac{3}{4}} u_{0}, (-\Delta)^{\frac{1}{4}} u_{1} \right) \right\|_{\gamma^{s}_{\eta, L^{2}} \times \gamma^{s}_{\eta, L^{2}}} + \frac{|c^{\prime}(t)|}{c(t)} \mathcal{E}_{w_k}(t) .
	\end{align*}
By integrating the previous inequality an applying Gr\"onwall's inequality, we obtain the bound
	\begin{multline*}
		\mathcal{E}_{w_{k}}(t)
		\leq \frac{4 \max(1, M)}{\min(1, \nu_{0})} e^{2 \nu_{0}^{-1} M} \left\| \left((-\Delta)^{\frac{3}{4}} u_{0}, (-\Delta)^{\frac{1}{4}} u_{1} \right) \right\|_{\gamma^{s}_{\eta, L^{2}} \times \gamma^{s}_{\eta, L^{2}}} \cdot \\
		\int_{0}^{t} |c_k(\tau) - c(\tau)| d\tau \exp\left( \int_{0}^{t} \frac{|c^{\prime}(\tau)|}{c(\tau)} d\tau \right) ,
	\end{multline*}
for $t \in [0, T_m)$. Consequently,
	\[ 
		\left. 
			\begin{gathered} 
				\nabla v_k(t) \to \nabla v(t) \\ 
				\pa_t v_k(t) \to \pa_t v(t) 
			\end{gathered} 
		\right\} 
		\quad 
		\text{in } L^{\infty}_{\loc}([0,T_m); L^2(\Rn)) \text{ as } k \to \infty .
	\]
Hence we obtain \eqref{EQ:small-S-convergence}, proving the continuity of $\Theta$.
\end{proof}

\begin{proof}[Proof of Lemma {\rm \ref{l:KConvexCompact}}]
As $\mathcal{K}$ is clearly convex, it suffices to show that $\mathcal{K}$ is compact. Now, let $(c_k)_{k \in \N}$ be a sequence in $\mathcal{K}$. Observe that
	\[ c_{k}(t) - c_{k}(t') = \int_{t'}^{t} c'_{k}(\tau) d\tau , \]
so that
	\[ |c_{k}(t) - c_{k}(t')| \leq s K \left[ \frac{1}{(T_m - t)^{1/s}} - \frac{1}{(T_m - t')^{1/s}} \right] , \]
for any $0 \leq t' < t < T_m$. As $1 / (T_m - \cdot)^{1/s}$ is uniformly continuous on any compact interval of $[0, T_m)$, the sequence $(c_k)_{k \in \N}$ is equicontinuous on that interval. Hence, by the Ascoli-Arzel\`a theorem, the sequence $(c_k)_{k \in \N}$ has a convergent subsequence $(c_{k_n})_{n \in \N}$ in $L^{\infty}_{\loc}([0, T_m))$ with limit $c \in L^{\infty}_{\loc}([0, T_m))$. To conclude the proof, it suffices to show that $c \in \mathcal{K}$. Clearly, $\nu_{0} \leq c(t) \leq M$ for every $t \in [0, T_m]$. Also, for any $0 \leq t' < t < T_m$ we have
	\[ |c(t) - c(t')| \leq s K \left[ \frac{1}{(T_m - t)^{1/s}} - \frac{1}{(T_m - t')^{1/s}} \right] . \]
Note that this already implies that $c \in \mathrm{Lip}_{\loc}([0,T_m))$ as $1 / (T_m - \cdot)^{1/s} \in \mathrm{Lip}_{\loc}([0,T_m))$. Whence, $c$ is almost everywhere differentiable on $[0, T_m)$. Let $t_0 \in [0, T_m)$ be a point where $c'(t_0)$ exists. For $h > 0$ small enough, we then have
	\[ \left| \frac{c(t_0 + h) - c(t_0 - h)}{2h} \right| \leq \frac{s K}{2h} \left[ \frac{1}{(T_m - t_0 - h)^{1/s}} - \frac{1}{(T_m - t_0 + h)^{1/s}} \right] , \]
so that by taking the limit $h \rightarrow 0^{+}$, we find
	\[ |c'(t_0)| \leq \frac{K}{(T_m - t_0)^{\frac{s + 1}{s}}} . \]
We may conclude that $c \in \mathcal{K}$, which completes the proof.
\end{proof}

We end this section with the following remark.

\begin{remark} \label{remark: improvement lower bound}
The lower bound given in \eqref{eq:Lower bound Tm} is strictly larger the one in \eqref{eq:classical lower bound} if and only if
	\[ \eta > 2 M \nu_{0}^{-1} + C_{s} \frac{\displaystyle{\int_{\Rn} e^{\eta |\xi|^{\frac{1}{s}}} \left[ |\xi|^{3} |(\mathcal{F} u_{0})(\xi)|^{2} + |\xi| |(\mathcal{F} u_{1})(\xi)|^{2} \right] d\xi}}{\displaystyle{\left( \int_{\Rn} \varphi(\|\nabla u_0\|^{2}_{L^{2}}) |\xi|^{3} |(\mathcal{F} u_{0})(\xi)|^{2} + |\xi| |(\mathcal{F} u_{1})(\xi)|^{2} \, d\xi \right)^{\frac{s + 1}{s}}}} , \]
with
	\[ C_{s} := \frac{\max(M, 1)}{\min(\nu_{0}, 1)} 2sL e^{2 \nu_{0}^{-1} M} \left( \frac{\nu_{0}^{3/2}}{4 L} \right)^{\frac{s + 1}{s}} . \]
\end{remark}


\begin{thebibliography}{10}
\bibitem{Arosio-Garavaldi} 
{\sc A.~Arosio, S.~Garavaldi},
{\em On the mildly degenerate Kirchhoff string}, 
Math.\ Methods Appl.\ Sci.\ \textbf{14} (1991), 177--195.



\bibitem{Arosio}
{\sc A.~ Arosio, S.~Spagnolo}, 
{\em Global solutions to the Cauchy problem for a 
nonlinear hyperbolic equation,} Nonlinear partial differential 
equations and their applications, Coll\`ege de France seminar, 
Vol. VI (Paris, 1982/1983), pp. 1--26, Res.\ Notes in Math., 109, 
Pitman, Boston, MA, 1984.


%

%
\bibitem{Bernstein}
{\sc S.~Bernstein}, 
{\em Sur une classe d'\'equations fonctionnelles aux 
d\'eriv\'ees partielles,} Izv.\ Akad.\ Nauk SSSR Ser.\ Mat.\ \textbf{4} (1940), 17--27.


%
\bibitem{Colombini}
   {\sc F.~Colombini, D.~Del~Santo, T.~Kinoshita}, 
\textit{Well-posedness of a hyperbolic equation with non-Lipschitz coefficients,} 
Ann. Scuola Norm. Sup. Pisa, Cl. Sci. (5), \textbf{1} (2002), 327--358.

%
 \bibitem{Dancona-Invent} 
{\sc P.~D'Ancona, S.~Spagnolo}, 
{\em Global solvability for the degenerate Kirchhoff equation with real analytic data}, 
Invent.\ Math.\ \textbf{108} (1992), 247-262.
 
%


%
\bibitem{Ghisi}
{\sc M.~Ghisi, M.~Gobbino}, 
{\em Kirchhoff equation from quasi-analytic to spectral-gap data}, 
Bull.\ London\ Math.\ Soc.\ \textbf{43} (2011), 374--385. 

\bibitem{Kajitani-Pisa}
{\sc K.~Kajitani, K.~Yamaguti}, 
{\em On global analytic solutions of the degenerate Kirchhoff equation},
Ann.\ Scuola\ Norm.\ Sup.\ Pisa\ Cl. Sci.\ (4)\ \textbf{21} (1994), 
279--297.


%
\bibitem{Kirchhoff}
{\sc G.~Kirchhoff}, 
{\em Vorlesungen \"uber Mechanik}, 
Teubner, Leibzig, 1876.


\bibitem{Manfrin-JDE} 
{\sc R.~Manfrin},  
{\em On the global solvability of Kirchhoff equation for non-analytic initial data}, 
J.\ Differential\ Equations\ \textbf{211} (2005), 38--60.

\bibitem{MR-Liouville}
{\sc T.~Matsuyama, M.~Ruzhansky}, 
{\em Global well-posedness of Kirchhoff systems}, 
J.\ Math.\ Pures\ Appl.\ \textbf{100} (2013), 220--240.

\bibitem{MR-JAM} 
{\sc T.~Matsuyama, M.~Ruzhansky}, 
{\em On the Gevrey well-posedness of the Kirchhoff equation}, 
J.\ Anal.\ Math.\ \textbf{137} (2019), 449--468.  


%
\bibitem{Nishihara}
{\sc K.~Nishihara}, 
{\em On a global solution of some quasilinear hyperbolic equation}, 
Tokyo\ J.\ Math.\ \textbf{7} (1984), 437--459.

%
\bibitem{Pohozhaev}
{\sc S.~I.~Pohozaev}, 
{\em On a class of quasilinear hyperbolic equations}, 
Math.\ USSR\ Sb.\ \textbf{25} (1975), 145--158.

%


%
\end{thebibliography}
\end{document}